\title{A random walk on $p$-groups with a symmetric perfect pairing}

\documentclass{article}

\usepackage{ifxetex,ifluatex}
\if\ifxetex T\else\ifluatex T\else F\fi\fi T%
  \usepackage{fontspec}
\else
  \usepackage[T1]{fontenc}

  \usepackage{blkarray, bigstrut}
  \usepackage[utf8]{inputenc}
  
  \usepackage{amsthm}
  \usepackage{thmtools}
  \usepackage{mathtools}

  \usepackage{lmodern}
  \usepackage{amssymb}
  \usepackage{amsfonts}
  \usepackage{tikz-cd}
  \usepackage{amsthm}
  \usepackage{xfrac}
  \usepackage{mathtools}
  \usepackage{multirow}
  \usepackage{mathrsfs}
  \usepackage{comment}

  \newtheorem*{corollary}{Corollary}

  \newcommand{\probP}{\text{I\kern-0.15em P}}
  \newcommand{\probE}{\text{I\kern-0.15em E}}
  
  \newcommand{\defeq}{\overset{\mathrm{def}}{=\joinrel=}}

  \numberwithin{equation}{section}

  \theoremstyle{remark}
  \newtheorem*{remark}{Remark}
  \theoremstyle{remark}

  \newtheorem*{observation}{Observation}
  \newtheorem*{definition}{Definition}

  \newcommand{\Q}{\mathbb{Q}}
  \newcommand{\Z}{\mathbb{Z}}

  \excludecomment{mysection}
  \excludecomment{mymysection}
  \excludecomment{maybeinclude}

    \newenvironment{mat}
    {
    \left[
\begin{array}}
{ \end{array}
\right]
}

\newcommand{\probM}{\mathcal{M}}

\newtheorem{Remark}{Remark}[section]

\usepackage{mdframed}

\theoremstyle{theorem}
\newtheorem{Claim}{Claim}

\newtheorem{theorem}{Theorem}
\newtheorem*{extratheorem}{Theorem}
\newtheorem{lemma}{Lemma}[theorem]
\newenvironment{ftheo*}
  {\begin{mdframed}\begin{theorem*}}
  {\end{theorem*}\end{mdframed}}

\newtheorem{Corollary}{Corollary}[theorem]

\newtheorem*{goal}{Goal}

  \usepackage{enumitem}
  \newtheorem*{recall}{Recall}

   \DeclareSymbolFont{bbold}{U}{bbold}{m}{n}
   \DeclareSymbolFontAlphabet{\mathbbold}{bbold}

   \theoremstyle{remark}

\newtheorem{question}{Question}

\fi

 \theoremstyle{definition}

 \newcommand{\comm}[1]{} 

 \usepackage{cite}
 \newcommand{\commm}[1]{}
 
 \newcommand{\chapter}{section}

  \usepackage{
  hyperref,
  cleveref,
  }

\author{Nikita Lvov} 
\begin{document} 
\maketitle

 \newcommand{\Elementg}{\mathbf{g}}
 \newcommand{\cccdot}{\hspace{0.15in}}
 \newcommand{\ccdot}{\,\cdot \,}
 \newcommand{\Gdual}{Hom \Big( ker(M_{n-1,n-1}), \Q_p/\Z_p \Big)}
 \newcommand{\VV}{V_{n-1}^{\perp}/V_{n-1}}
 \newcommand{\probMM}{SymHaar}
 \newcommand{\probMMM}{ASHaar}
 \newcommand{\Ker}{\widetilde{ker}}

\newcommand{\comma}{,}
\newcommand{\mcH}{\mathcal{H}}
\newcommand{\mcG}{\mathcal{G}}
\newcommand{\sdot}{\, \cdot \,}
\newcommand{\qtensor}{\otimes \Q_p}
\newcommand{\somespace}{\newline \newline \noindent}

\renewcommand{\mcG}{X_S}

\abstract{The kernel of a random symmetric p-adic matrix is a random abelian group, equipped with a symmetric pairing. If we consider not only the matrix but also its top-left corners, we get a process valued in isomorphism classes of abelian groups, equipped with such a pairing. We show that when the matrix is Haar random, this process is a Markov chain, generated by an operator that we explicitly describe. We will also prove that this operator is reversible with respect to a Cohen-Lenstra type measure. }
\section{Introduction}
The main purpose of this paper is to generalize some of the results of \cite{arxiv} to symmetric matrices. We first recall the case of $p$-adic matrices without symmetry.
\newline\newline \noindent
A $k \times k$ matrix $M$, with coefficients in $\Z_p$, gives rise to a homomorphism:
\begin{equation}
\label{eqn:homomorphism}
(\Q_p /\Z_p)^k \xrightarrow{M} (\Q_p / \Z_p)^k
\end{equation}
Up to automorphisms of $(\Q_p / \Z_p)^k$, this homomorphism is uniquely specified by the structure of its kernel, as an abelian group. This is a consequence of the existence of the Smith Normal Form. Thus, $ker(M)$ is a natural invariant.

\begin{question}
What is the distribution of $ker(\probM)$, when $\probM$ is a large random matrix?
\end{question}

A number of different distributions for $\probM$ can be considered, but there is one that is most natural. Indeed, a pleasant feature of $\Z_p$, as compared to $\Z$, is that $\Z_p$ is a pro-finite ring, and hence carries a natural probability measure. In the literature, this is alternatively called the \textit{uniform} measure, or the \textit{Haar} measure. For this measure, we have the following asymptotic theorem:

  \begin{extratheorem} \cite{FriedmanWashington}
  \label{thm:matricesandcl}
  If $\mathcal{M}_{k,k}$ is a random matrix whose entries are independent Haar distributed elements of $\Z_p$, $ker(\mathcal{M}_{k,k})$ is finite almost surely and:
  \begin{equation}
  \label{eqn:probg}
  \lim_{k \rightarrow \infty}\probP\Big( ker(\mathcal{M}_{k,k}) \cong G \Big) \propto \frac{1}{|Aut(G)|}
  \end{equation}
  \end{extratheorem}

\noindent
We can consider many other distributions on matrices. This leads to many interesting variants of \textit{Question 1}. The reader is referred to \cite[Sections 3.1 and 3.2]{WoodICM} for a survey. See also \cite{Van_Peski_2021}, \cite{CheongYu} and \cite{Van_Werde_2024} for other interesting random matrix models, not covered  in \cite{WoodICM}. 

\begin{remark}
Usually the question and theorem above are phrased in terms of 
\[
coker(M)\defeq coker\left(\Z_p^k \xrightarrow{M} \Z_p^{k}\right)
\]
rather than $ker(M)$. Either group is determined by the other, and when $M$ in non-singular, both groups are canonically isomorphic via multiplication by $M$. We will work with the kernel as that choice simplifies some of our arguments later on, such as the proof of \autoref{claim:splittingisomorphism}.
\somespace
Both choices are equivalent.

\end{remark}

\subsection{The present paper}

In \cite{arxiv}, we undertook a deeper study of the above result of Friedman and Washington. We will briefly recall the details:

\paragraph{A random process valued in abelian $p$-groups.} Let $Haar$ be an infinite Haar random matrix. Let $Haar_{n,n}$ be its top $n \times n$ left corner. $ker(Haar_{n,n})$ is a random process valued in abelian $p$-groups. In \cite{arxiv}, we studied this process and found that:

\begin{itemize}
\item $ker(Haar_{n,n})$ is a reversible Markov chain; moreover, the Markov chain can be described as a random walk on an explicit weighted graph, whose edges are parametrized by certain exact sequences.
\item The transition operator has a simple explicit description:
\begin{itemize}
\item Take a random $\Z_p$-extension, then mod out by the span of a Haar random element.
\end{itemize}
\item The spectral problem for this Markov chain has an explicit solution.
\end{itemize}

\paragraph{Symmetric Matrices}
In this paper, we would like to generalize some of the results of \cite{arxiv} to a different random matrix model:

\begin{definition}
Define $SymHaar$ to be the random matrix obtained by restricting the Haar measure on infinite matrices to infinite symmetric matrices, and let $SymHaar_{n,n}$ be the top-left $n\times n$ corner.
\end{definition}
\noindent
Regarded as a process on finite abelian $p$-groups, $ker(SymHaar_{n,n})$ does \textit{not} have the Markovian property. However, we will discover that we can recover the Markovian property if take into account the extra information encoded by the following pairing on $ker(M)$:

\begin{equation}
\label{eqn:pairing}
ker(M) \times ker(M) \rightarrow \Q_p/\Z_p \, : \, (v,w) \mapsto
v^{T}Mw \, \mod \Z_p
\end{equation}

\noindent
This motivates the following definition:

\begin{definition}
Given a symmetric matrix $M$, we define $\Ker(M)$ to be the abelian group $ker(M)$, together with the pairing (\ref{eqn:pairing}), \textit{up to isomorphism}.
\end{definition}


\begin{Remark}
\label{rem:perfect}
For an equivalent definition of this pairing, see (\ref{eqn:isokernel}), which originally appeared as \cite[(2)]{BKLPR}. It is a direct consequence of this latter definition that the pairing (\ref{eqn:pairing}) is perfect whenever $M$ is non-singular. Of course, the fact that the pairing is perfect can also be checked directly.
\end{Remark}

\begin{Remark}
In the case when $M$ is non-singular, we can identify $ker(M)$ and $coker(M)$ via multiplication by $M$. Via this identification, we can pull back the pairing on $ker(M)$ to $coker(M)$. This pairing will have the form:
\end{Remark}
\begin{equation}
\label{eqn:pairingtwo}
coker(M) \times coker(M) : (v,w) \mapsto v^T M^{-1}w \mod \Z_p 
\end{equation}
\noindent
This is the pairing used in \cite{Woodetal}. Note that it follows, for example, from \cite[Theorem 2]{Woodetal} that a Haar random symmetric matrix is non-singular with probability $1$.

\paragraph{\textbf{Main results.}}
The first main result of this paper is the following:

\begin{theorem}
\label{thm:symmarkovchain}
$\Ker(SymHaar_{n,n})$ is a Markov chain.
\end{theorem}

\begin{mysection}
Thus, the state space consists of isomorphism classes of abelian $p$-groups endowed with a symmetric pairing.
\end{mysection}

\noindent
In the next paragraph, we will give a description of the generator of the Markov chain. For now, we make a few remarks.
\newline

\noindent
In the same way that \cite[Theorem 1.1]{arxiv} provides a dynamical generalization of  the result of Friedman and Washington, \autoref{thm:symmarkovchain} can be regarded as a dynamical version of the following theorem of \cite{Woodetal}:

\begin{extratheorem}\cite[Theorem 2]{Woodetal}
Let $G$ be a finite abelian $p$-group of rank $r$ with a perfect symmetric pairing $<\cdot,\cdot>$. Then:
\[
\probP\Big(\Ker(SymHaar_{n,n}) \cong (G,<\cdot,\cdot>)\Big)= \]
\[\frac{
\prod_{j=n-r+1}^{n} (1-p^{-j}) \prod_{i=1}^{\lceil (n-r)/2 \rceil}(1-p^{1-2i})
}{
\# G
\# Aut(G,<\cdot,\cdot>)
}
\]
where $Aut(G,<\cdot,\cdot>)$ is the group of automorphisms of $G$ that preserve the pairing $<\cdot,\cdot>$. 

\end{extratheorem}

\begin{remark}
As opposed to \cite{arxiv}, we do not study the spectral problem for the Markov chain of \autoref{thm:symmarkovchain} in this paper.
\end{remark}

\paragraph{\textbf{Explicit form of the generator of the Markov chain in the symmetric case.}}
\label{sec:defdeltas}

Let $(G,<,>)$ be a finite abelian group, equipped with a perfect symmetric pairing. $\Delta_{S}(G,<,>)$ is defined as the output of the following random algorithm:

\begin{enumerate}
\item Pick an element $g \in G$ uniformly at random.
\item Pick a symmetric pairing on $G \times \Q_p$ \textit{uniformly\footnote{What this means is defined below.}} at random from all choices that satisfy:
\begin{itemize}
\item[(A)] The pairing, restricted to $G$, is the original pairing on $G$.
\item[(B)] $(g,1) \in (g,1)^{\perp}$ with respect to the pairing.
\end{itemize}
\end{enumerate}

Then $\Delta_{S}(G,<,>)$ is 
\begin{equation}
\label{eqn:defdeltas}
(g,1)^{\perp}/(g,1)
\end{equation}

This is a finite group almost surely.

\begin{remark}
We need to say a few words about how we choose the pairing. There is a natural way to do this.
\begin{itemize}
\item We already know the pairing on $G$.
\item There can be no cross-terms, as $G$ is finite, and hence \[Hom\Big(G,Hom(\Q_p,\Q_p/\Z_p)\Big)=0\] 
\end{itemize}
Hence, we need only choose a pairing on $\Q_p$. Every pairing on $\Q_p$ is of the form:
\[
<q_1, q_2> \defeq q q_1 q_2 \mod \Z_p
\] 
Thus, to determine the pairing, we need to choose $q \in Q_p$. Condition (B) implies that $q$ must lie in a compact subset of $\Q_p$. Finally, compact subsets of $\Q_p$ carry a natural probability measure. 
\end{remark}

\begin{remark} Concretely, if the pairing on $\Q_p$ is of the form \begin{equation}
\label{eqn:qpairing}
<q_1,g_2> \defeq q q_1 q_2\mod \Z_p,
\end{equation}
 then
\[
(g,1) \in (g,1)^{\perp} \Leftrightarrow q \equiv -<g,g> \mod \Z_p.
\] 
\end{remark}


\noindent
We now come to the main theorem of this paper, for symmetric matrices.

\begin{theorem}
\label{eqn:matrixkernelequation}
\label{thm:maintheorem}
Suppose $M$ is symmetric and non-singular. Then $\Delta_S$ satisfies
   \[
   \Delta_S \Big[ \Ker(M) \Big]=
   \Ker 
   \begin{mat}{ccc|c}
   & & & *\\
   & M& &\vdots\\
   & & & *\\ \hline
   * & \hdots & * & *
   \end{mat} \hspace{1in} \text{ for all } \, M
   \]
   where the entries denoted by $*$ are uniformly random variables, that are independent, subject to the symmetry condition.
\end{theorem}

\begin{Corollary}
\label{thm:askerneltheorem}
$\Delta_{S}$ is the generator of the Markov chain $\Ker(SymHaar_{n,n})$.
\end{Corollary}

\subsubsection{Reversibility of the Markov chain}

In \S \ref{sec:reversibility} of this paper, we will prove the following theorem:

\begin{theorem}
The Markov chain generated by $\Delta_S$  on the state space $X_S$ is reversible with respect to the probability measure:
\begin{equation}
\label{def: measure s}
\frac{
\prod_{i=1}^{\infty}(1-p^{1-2i})
}{
\# G
\# Aut(G,<\cdot,\cdot>)
}
\end{equation}
\end{theorem}
\noindent
A short outline of the proof is given in \S \ref{sec: overview reversibility}. The main idea is to construct a weighted graph on $X_S$, such that $\Delta_S$ is a random symmetric walk on this graph.

\begin{remark}
We note that the transitivity and recurrence of the Markov chain follow directly from the statement of \cite[Theorem 2]{Woodetal}, together with \autoref{thm:maintheorem}.
\end{remark}
\noindent
Section \ref{sec:reversibility} can be read independently of Section \ref{sec:markovchain}.

\subsubsection{Other work on $\Ker(SymHaar_{n,n})$}
\label{sec:otherwork}
In an unpublished pre-print, \cite{Maples2}, Maples introduced a Markov chain that models $ker(SymHaar_{n,n}) \otimes \mathbb{F}_p$.
\newline
\newline
\noindent
In recent work, \cite{shen}, Shen and Van Peski study the analogous random process for anti-symmetric and Hermitian matrices. Their approach uses the theory of symmetric polynomials and is in certain complementary to ours. In fact, the combination of both the symmetric function point of view, and the group-theoretic point of view, often proves to be highly insightful, as will be shown, for example, in a forthcoming appendix to \cite{arxiv}.

\section{Outline of the proof of \autoref{thm:maintheorem}}
\label{sec:markovchain}

\subsection{Notation}

We briefly collect here for reference all the notation that we will use in \S \ref{sec:markovchain}:
\newline
\newline $X_S$ denotes the set of groups equipped with a perfect symmetric pairing.
\newline $(G,<,>)$ denotes an element of $X_S$
\newline $g$ denotes an element of $G$.
\newline $\Delta_S$ denotes the generator of a Markov chain defined on $X_S$ in \S\ref{sec:defdeltas}.
\newline $\Ker(M)$ denotes an element of $X_S$ associated to a symmetric matrix $M$ with entries in $\Z_p$.
\newline $\probM$ denotes a random matrix.
\newline $M$ denotes a deterministic matrix.
\newline $M_{k,k}$ denotes the top-left $k \times k$ minor of a given matrix.

\subsection{The group associated with a matrix $M$}
The purpose of this section is to give another perspective on $\Ker(M)$, that will be particularly useful for our proof of \autoref{thm:maintheorem}.
\newline
\newline
A symmetric matrix $M_{n,n}$ with entries in $\Q_p$ determines a symmetric pairing on $\Q_p^n$:
\[
<\, \, , \, \,>: \, \Q_p^n \times \Q_p^n \rightarrow \Q_p / \Z_p 
\]
\begin{equation}
\label{eqn:defpairing}
<v_1,v_2> \defeq v_1^T M v_2 \mod \Z_p
\end{equation}

\begin{remark}
If all the entries of $M$ lie in $\Z_p$, then $\Z_p^n \subset (\Z_p^n)^{\perp}$. If the matrix is non-singular, then the pairing is perfect. Furthermore, in this case,
\begin{equation}
\label{eqn:isokernel}
\Ker(M) \cong (\Z_p^n)^{\perp} \Big/ \Z_p^n
\end{equation}
Indeed, one can verify that the right hand side agrees with the definition of $\Ker$ given in the introduction. (\ref{eqn:isokernel} also appears as \cite[(2)]{BKLPR}. 
\end{remark}

\noindent
It also clear from (\ref{eqn:isokernel}) that the pairing on $\Ker$ is perfect, as claimed in the \autoref{rem:perfect}. In the next three sections, we will prove \autoref{eqn:matrixkernelequation}.

\begin{mysection}
\begin{Claim}
\label{claim:defkerM}
When $M$ is non-singular, the pairing will be perfect and the group will be finite.
\end{Claim}

\begin{remark}
This differs slightly from the usual construction, which considers the co-kernel of $M$ \cite{Wood}, but both constructions are equivalent, for non-singular $M$.
\end{remark}
\end{mysection}

\subsection{The relation between $\Ker(M_{n,n})$ and $\Ker(M_{n-1,n-1})$ }
\label{sec:relation}

\begin{definition}
We define 
\[
V_k \defeq \Z_p^{k}  \subset \Q_p^{n}
\]
to be the subgroup of $\Q_p^n$ spanned over $\Z_p$ by the first $k$ coordinate vectors of $\Z_p^k$.
\end{definition}

\noindent
Observe that we have the following inclusions:
\begin{equation}
\label{eqn:nested}
V_1 \subset \cdots \subset V_{n-1} \subset V_n \subset V_n^{\perp}
 \subset V_{n-1}^{\perp} \subset \cdots \subset V_1^{\perp}
\end{equation}

\noindent
Central to our argument is the following generalization of (\ref{eqn:isokernel}).
\begin{lemma}
\label{claim:splittingisomorphism}
Suppose $1 \leq k \leq n$. If $M_{k,k}$ is non-singular, then there is a canonical isomorphism of groups equipped with a symmetric pairing:
\begin{equation}
\label{eqn:splittingisomorphism}
V_k^{\perp} \Big/ V_k \cong \Ker(M_{k,k}) \times \Q_p^{n-k} 
\end{equation}
where the pairing on the right hand side is the sum of the pairing on $\Ker(M_{k,k})$, and some symmetric pairing on $\Q_p^{n-k}$. The pairing on the left hand side is the pairing induced by $M_{n,n}$.
\end{lemma}

\begin{proof}
We will defer the proof of \autoref{claim:splittingisomorphism} to \S \ref{sec:splittingisomorphism}.
\end{proof}

\begin{observation}
We make the simple observation that $V_n/V_{n-1}$ is isomorphic to $\Z_p$. $V_n/V_{n-1} \cong \Z_p$ is generated by the $n^{th}$ coordinate vector in $\Z_p^n$:
\[
\begin{array}{ccccccc}
e_n \defeq [ & 0 & 0 & \hdots   &  0 &  1 & ] \\
\end{array}
\]
\end{observation}

\noindent
As an element of $X_S$, $\Ker(M_{n,n}) = V_n^{\perp}/V_{n}$ can be expressed as:
\begin{equation}
\label{eqn:envnisomorphism}
V_n^{\perp}/V_{n} \cong e_n^{\perp} / e_n
\hspace{0.1in} \text{ in} \hspace{0.1in} V_{n-1}^{\perp}/V_{n-1}  
\end{equation}

\noindent
Recall that, by \autoref{claim:splittingisomorphism}, $V_{n-1}^{\perp}/V_{n-1} \cong \Ker(M_{n-1,n-1}) \times \Q_p$.
We will now combine  \ref{eqn:splittingisomorphism} and \ref{eqn:envnisomorphism} in order to relate $ker(M_{n-1,n-1})$ and $ker(M_{n,n})$. For this, we need to understand the image of $e_n$ in $\Ker(M_{n-1,n-1}) \times \Q_p$.
\paragraph{The image of $e_n$ in $\Ker(M_{n-1,n-1}) \times \Q_p$}
There are two restrictions on the image of $e_n$, given below:

\begin{itemize}
\item The image of $e_n$ in $ker(M_{n-1,n-1}) \times \Q_p$ must be of the form $(g,1)$ for some $g \in ker(M_{n-1,n-1})$.
\begin{proof}
This is the content of \autoref{extrastatement} below.
\end{proof}

\item $(g,1)$ must satisfy $(g,1) \in (g,1)^{\perp}$, which implies the compatibility condition:
\begin{equation}
\label{eqn:compatibility}
\Big<g,g\Big>_{M_{n-1,n-1}} + \Big<1,1\Big>_{\Q_p} = 0
\end{equation}
\end{itemize}

\paragraph{Recipe to get $\Ker(M_{n,n})$ from $\Ker(M_{n-1,n-1})$}

We summarize:

\begin{lemma}
\label{lem:recipe}
$M_{n,n}$ determines the following data, when $M_{n-1,n-1}$ is non-singular.
\begin{itemize}
\item $\Ker(M_{n-1,n-1})$,
\item An element $g \in ker(M_{n-1,n-1})$
\item A pairing on $\Q_p$ that satisfies (\ref{eqn:compatibility}).
\end{itemize}
Given this data, 
\[
\Ker(M_{n,n}) \cong (g,1)^{\perp}/(g,1) \hspace{0.1in}\text{ in }
\hspace{0.1in} \Ker(M_{n-1,n-1}) \times \Q_p
\]

\end{lemma}

\subsection{Deduction of \autoref{thm:maintheorem}}

\begin{recall} In order to prove \autoref{thm:maintheorem}, we need to understand the distribution of \begin{equation}
\label{eqn: random matrix}
\Ker
 \begin{mat}{ccc|c}
   & & & *\\
   & M& &\vdots\\
   & & & *\\ \hline
   * & \hdots & * & *
   \end{mat},
\end{equation}
given a fixed $M$. We have used the same notation as in the statement of \autoref{thm:maintheorem}. 
\end{recall}
\noindent
We achieve this by applying \autoref{lem:recipe} to the matrix:
\[
 \begin{mat}{ccc|c}
   & & & *\\
   & M& &\vdots\\
   & & & *\\ \hline
   * & \hdots & * & *
   \end{mat}
\]
This determines a random element $g \in ker(M)$ and a random pairing on $\Q_p$.
We identify the pairing on $\Q_p$ with an element $q \in \Q_p$ via (\ref{eqn:qpairing}). Hence, to determine the distribution of \ref{eqn: random matrix}, it is sufficient to determine the distribution of the pair $(g,q)$.
\newline
\newline
\noindent
The following theorem gives sufficient information about the distribution of $g$ and $q$ to conclude \autoref{thm:maintheorem}.
\begin{theorem}
\label{thm:distribution}
In the case of the matrix \[ \begin{mat}{ccc|c}
   & & & *\\
   & M& &\vdots\\
   & & & *\\ \hline
   * & \hdots & * & *
   \end{mat},\]
   the joint distribution of $g$ and $q$ satisfies the following properties:
\begin{itemize}
\item[(A)] $g$ is a uniformly random element of $ker(M)$.
\item[(B)] The joint distribution of $g$ and $q$ is invariant by the operation
\[
q \rightarrow q+a
\]
for any $a \in \Z_p$,
\end{itemize}
\end{theorem}
\noindent
We defer the proof to \S \ref{sec: distribution theorem proof}.

\begin{corollary}
\autoref{thm:maintheorem} follows.
\end{corollary}

\begin{proof}(of Corollary)
Indeed, \autoref{thm:maintheorem} follows because there is a unique distribution on pairs $(g,q)$ that satisfies conditions $A$, $B$, and the condition:
\[
q \equiv -<g,g> \mod \Z_p.
\]
Under this distribution, $g$ is uniformly random and $q$ is a uniformly random lift of $-<g,g>$ to $\Q_p$. \autoref{thm:maintheorem} now follows by the last line of \autoref{lem:recipe}.
\end{proof}

\noindent
To complete the proof of \autoref{thm:maintheorem} it remains to show the following unproven statements: \autoref{claim:splittingisomorphism}, \autoref{extrastatement} and \autoref{thm:distribution}.

\subsection{Proofs of unproven claims}
\subsubsection{Proof of \autoref{claim:splittingisomorphism}}
\label{sec:splittingisomorphism}

\paragraph{The structure of $V_k^{\perp}/V_k$ as a group.}

\begin{proof} 
First, we show that the projection if $V_k^{\perp}$ onto the last $n-k$ coordinates,
\[
V_k^{\perp} \rightarrow \Q_p^{n-k},
\]
is surjective. To do this, we will prove that
\[
(V_k \otimes \Q_p)^{\perp} \rightarrow \Q_p^{n-k}
\]
is surjective, a stronger statement. As $M_{k,k}$ is non-singular, on the left we have a vector space of dimension $n-k$. Thus, to prove surjectivity, we need to prove that the kernel is trivial. But the kernel must lie in
\[
V_k \otimes \Q_p \cap (V_k \otimes \Q_p)^{\perp} = \emptyset
\]
\newline
\newline
It follows that 
\[
V_k^{\perp}/V_k \rightarrow \Q_p^{n-k}
\]
is also surjective. The kernel of this map is 
\[
V_k \otimes \Q_p \cap V_k^{\perp}  \cong ker(M_{k,k}).
\] Thus, we get an exact sequence:
\[
0\rightarrow ker(M_{k,k}) \rightarrow V_k^{\perp}/V_k \rightarrow \Q_p^{n-k} \rightarrow 0
\]
Because $ker(M_{k,k})$ is finite, the sequence splits canonically. 
\end{proof}

\paragraph{The pairing on $V_k^{\perp}/V_k$}
\begin{Claim}
\label{claim:pairing}
Any bilinear pairing on 
\[
ker(M_{k,k}) \times \Q_p^{n-k}
\]
must split as a sum of a pairing on the left factor and a pairing on the right factor. 
\end{Claim}

\begin{proof}
There can be no cross-terms, as the only homomorphism from $ker(M_{k,k})$ to the dual of $\Q_p^{n-k}$ is trivial. Indeed, $ker(M_{k,k})$ is finite while the dual of $\Q_p^{n-k}$ is isomorphic to $\Q_p^{n-k}$.
\end{proof}

It remains to show the following
\begin{Claim}
\label{claim:pairingrestriction}
The pairing (\ref{eqn:defpairing}) induced by $M_{n,n}$ on 
\begin{equation}
\label{eqn:anothersplitting}
V^{\perp}_k / V_{k} \cong ker(M_{k,k}) \times \Q_p^{n-k}
\end{equation}
restricts to the pairing induced by $M_{k,k}$ on $ker(M_{k,k})$.
\end{Claim}

\begin{proof}
The image of
\[
ker(M_{k,k}) \hookrightarrow V^{\perp}_k / V_{k}
\]
is non-zero only on the first $k$ coordinates. The restriction of the pairing induced by $M_{n,n}$ to the first $k$ coordinates is given by the top-left $k\times k$ corner of $M_{n,n}$, i.e. the matrix $M_{k,k}$. 
\end{proof}

%

\paragraph{The image of $e_n$:}

Recall that $e_n$ is defined to be the $n^{th}$ coordinate vector in $\Z_p^n$.
\begin{lemma}
\label{extrastatement}
The image of $e_n$ in $ker(M_{n-1,n-1}) \times \Q_p$ must be of the form $(g,1)$ for some $g \in ker(M_{n-1,n-1})$.
\end{lemma}

\begin{proof}
By construction, the map from $V_{n-1}^{\perp}/V_{n-1} \cong ker(M_{n-1,n-1}) \times \Q_p$ unto the $\Q_p$-factor is projection of $V_{n-1}^{\perp}/V_{n-1}$ unto the last coordinate. The last coordinate of $e_n$ is $1$.
\end{proof}

\subsubsection{Proof of \autoref{thm:distribution}}
\label{sec: distribution theorem proof}

In this section, we wish to determine the joint distribution of $g$ and the pairing on $\Q_p$. We first make two important observations.

\begin{itemize}
\item $g$ is determined uniquely by the dual of $g$, denoted as $g^{\vee}$.
\item The pairing on $\Q_p$ is determined uniquely by the dual of $1$, denoted as $q$.
\end{itemize}

To prove \autoref{thm:distribution}, it suffices to show:
\begin{lemma}
\label{lem:modifiedlemma}
The joint distribution of $(g^{\vee},q)$ satisfies the following two conditions:
\begin{itemize}
\item $g^{\vee}$ is a uniformly random element of the dual of $ker(M_{n-1,n-1})$.
\item The joint distribution of $(g^{\vee},q)$ is invariant under $q \rightarrow q+a$, for any $a 
\in 
\Z_p$.
\end{itemize}
\end{lemma}

\noindent
To prove \ref{lem:modifiedlemma}, we make the observation that 
\[
(g,1)^{\vee}= (g^{\vee},q)
\]
and recall that $(g,1)$ is defined as the image of $e_n$ in $\VV$. Hence, $(g,1)^{\vee}$ is the pull-back of $e_n^{\vee}$ to $\VV$. But because
\[
e_n^{\vee} = 
\begin{mat}{ccc|c}
   & & & *\\
   & M& &\vdots\\
   & & & *\\ \hline
   * & \hdots & * & *
   \end{mat} e_n,
\] 
$e_n^{\vee}$ is a Haar random vector in $\Z_p^n$.

\paragraph{Calculating the pull-back of to $e_n^{\vee}$ to $V_n^{\perp} / V_n$}
$V_{n-1}^{\perp}/V_{n-1}$ is a random subgroup of $\sfrac{\Q_p^n}{V_{n-1}}$. It is not independent of $e_n^{\vee}$. Hence, care is needed when taking the pull-back. We will use the following facts:

\begin{itemize}
\item The subgroup \[ker(M_{n-1,n-1}) \hookrightarrow V_{n-1}^{\perp}/V_{n-1}\] is a \textit{fixed} subgroup of $\sfrac{\Q_p^n}{V_{n-1}}$.
\item $V_{n-1}^{\perp}/V_{n-1}$ is independent of the bottom right entry of $\probM_{n,n}$.
\item For $a \in \Z_p$, the pull-back of
\[
\begin{mat}{ccccc}
0&0&\hdots&0&a\\
\end{mat} \in \, Hom\Big( \sfrac{\Q_p^n}{V_{n-1}} , \sfrac{\Q_p}{\Z_p}  \Big)
\]
to $V_{n-1}^{\perp}/V_{n-1} \cong \ker(M_{n-1,n-1}) \times \Q_p$ is $(0,a)$.
\item[] \begin{proof}
We prove the last statement; the other two statements can be checked directly. The last statement follows because the composition
\begin{equation}
\label{eqn: map just defined}
V_{n-1}^{\perp}/V_{n-1} \cong ker(M_{n-1,n-1}) \times \Q_p \rightarrow \Q_p
\end{equation}
is projection unto the last coordinate.
\end{proof}
\end{itemize}

\begin{lemma}
\label{lem:lemmaA}
$g^{\vee}$ is a uniformly random element in the dual of $ker(M_{n-1,n-1})$
\end{lemma}

\begin{proof}
Firstly, we observe that, although $\VV$ can vary, the kernel of (\ref{eqn: map just defined}) is a fixed subgroup of $\Q_p^n / \Z_p^n$.
\newline
\newline
$g^{\vee}$ is the pull-back of $e_n^{\vee}$ to $ker(M_{n-1,n-1})$
$e_n^{\vee}$ is a Haar random element in the dual of $(\Q_p/\Z_p)^n$. The pull-back of a Haar random element in the dual of $(\Q_p/\Z_p)^n$ to a fixed subgroup is Haar random. The Haar measure on a finite group is the uniform measure.
\end{proof}

\begin{lemma}
\label{lem:qtranslation}
The distribution of $(g^{\vee},q)$ is invariant under the action of:
\[
q \rightarrow q+a
\]
for any $a \in \Z_p$.
\end{lemma}

\begin{proof}
The operation \begin{equation}
\label{eqn:matrixsum}
\probM_{n,n} \rightarrow \probM_{n,n} +
\begin{mat}{ccc|c}
0&\hdots&0&0\\
\vdots&\ddots&\vdots&\vdots\\
0&\hdots&0&0\\ \hline
0&\hdots&0&a\\
\end{mat}
\end{equation} does not change the distribution of $\probM_{n,n}$. This operation:
\begin{itemize}
\item Changes $e_n^{\vee}$ to 
\[
e_n^{\vee} + 
\begin{mat}{ccccc}
0&0&\hdots&0&a\\
\end{mat}
\]
\item Does not change
$
V_{n-1}^{\perp}
$
\end{itemize}
Because of the second point, the difference of the pull-backs of 
\[
e_n \hspace{0.1in} \text{ and } \hspace{0.1in} e_n^{\vee} + \begin{mat}{ccccc}
0&0&\hdots&0&a\\
\end{mat}
\]
to $V_{n-1}^{\perp}/V_{n}$, is the pull-back of 
\[
\begin{mat}{ccccc}
0&0&\hdots&0&a\\
\end{mat}
\]
$V_{n-1}^{\perp}/V_{n-1}$. By the statement directly preceding \autoref{lem:lemmaA}, this is $(0,a)$. Thus $(g^{\vee},q)$ changes to $(g^{\vee},q+a)$. Thus, since the operation (\ref{eqn:matrixsum}) does not change the distribution of $\probM_{n,n}$. we conclude that the joint distribution of $(g^{\vee},q)$ is likewise invariant under 
\[
q \rightarrow q+a.
\]
This proves \autoref{lem:qtranslation}.
\end{proof}

\begin{mysection}
(((Ideas to pursue: over real numbers, random tensors, multiple dimensions)))
((( Functoriality of process as coming from functoriality of edges. )))
((((Entropy? - process as gradient flow))))
(Babushkina peredacha pro okeany: Shokiruyushie gipotezy)
(Peredacha bylo v den' razgovora s Alexom, v voskresenie, 7 aprelya. Pro pogrujenie materikov pod vodu po-moemu...)
\end{mysection}

\section{Proof of Reversibility}
\label{sec:reversibility}

The aim of the remainder of this paper is to show that the Markov chain $\Delta_S$ is reversible with respect to the measure (\ref{def: measure s}). This section can be read independently of the rest of this paper.

\subsection{Overview}
\label{sec: overview reversibility}

Recall that we define $\mcG$ to be the set of isomorphism classes of finite abelian $p$-groups that are equipped with a perfect symmetric pairing. $\Delta_S$ generates a Markov chain on the state space $\mcG$. We now briefly outline the proof of reversibility.
\somespace 
We first recall the definition of $\Delta_S$ and provide some equivalent formulations. We then observe that we can realize the process\footnote{This is just the natural way in which the definition of the process as Markov chain suggests an interpretation as a random walk on a \textit{directed} weighted graph.} as a random walk on a \textit{directed} weighted graph $\Gamma$. We have some freedom in choosing the weights, as the total weight of edges leaving a vertex can be chosen to be any non-vanishing measure $\mu$ on $\mcG$. We show that the set of directed edges of $\Gamma$ is in bijection with isomorphism classes of \textit{pairs} of certain exact sequences, We denote the set of pairs of such sequences as $\mcH$. Elements of $\mcH$ are pairs of exact sequences:
\begin{equation*}
0\rightarrow \Z_p\xrightarrow{v_1} H \rightarrow H/v_1 \rightarrow 0 \hspace{0.1in}\text{ and }\hspace{0.1in} 0\rightarrow \Z_p\xrightarrow{v_2} H \rightarrow H/v_2 \rightarrow 0
\end{equation*}
together with perfect pairings:
\[<\,,\,>_1\,:\, H/v_1 \times H/v_1 \rightarrow \Q_p/\Z_p \hspace{0.1in}\text{ and } \hspace{0.1in} <\,,\,>_2 \,:\, H/v_2 \times H/v_2 \rightarrow \Q_p/\Z_p.\]
whose pull-backs to $H$ satisfy a certain compatibility condition (namely, condition \ref{eqn: relation}) with $v_1$ and $v_2$. Such a pair of sequences above corresponds to a directed edge, that starts at $(H/v_1,<,>_1)$ and that ends at $(H/v_2,<,>_2)$. We observe that isomorphism classes in $\mcH$ admit a simple involution $\sigma$ that reverses $(H/v_1,<,>_1)$ and $(H/v_2,<,>_2)$. Therefore, this involution reverses the endpoints of the corresponding directed edge of $\Gamma$. Furthermore, we observe that there is a choice of $\mu$ such that the weight of the directed edges is invariant under this involution.  We conclude that this choice of $\mu$ allows us to realize the process as a symmetric random walk on an undirected weighted graph. Futhermore, the symmetric random walk is reversible with respect to the above choice of $\mu$.

\newcommand{\appendixone}{
\subsubsection{Comparison with the Markov chain generated by $\Delta_0$ (the analogue of $\Delta_S$ for matrices without symmetry)}

Recall that the Markov chain generated by $\Delta_0$ can be realized as a walk on a random graph whose edges are in bijection with isomorphism classes of exact sequences ($0 \rightarrow \Z_p \rightarrow H \rightarrow G \rightarrow 0$), endowed with the weight \[1\,/ \,\#Aut(\Z_p \rightarrow H)\]\footnote{We remark that this bipartite graph is simple, i.e. any two vertices are connected by at most one edge. Also we note that two vertices are connected if and only if the associated partitions \textit{interlace}.}.
\somespace
The operator $\Delta_0$ corresponds to taking two steps on this graph.
\somespace
A natural analogue would be to consider a bipartite graph whose edges are of the form
\[
...
\]
In this case, we would need the weight
\[
...
\]
to get the proper equilibrium measure. However, the process corresponding to this graph would have the property that the quotient does not depend on the extension, and the second pairing does not depend on the first. (Both the quotient and the pairing only depend on the structure of $H$.)
\somespace
The compatibility ensures that the pairings are related, and that the quotient depends on the extension.
}

%
%

\subsection{Various equivalent definitions of $\Delta_S$}

In this section, we will recall the definition of $\Delta_S$ and we will give two other equivalent reformulations, the first of which will be very important for the sequel. The purpose of the last reformulation is to point out the connection with the process studied in \cite{arxiv}.

\subsubsection{Original definition of $\Delta_S$}
Recall our original definition of the process. Given $G\in \mcG$, we pick a random element $g \in G$. Then pick a random pairing on $G \times \Z_p$, that
\begin{itemize}
    \item  restricts to the original pairing on $G$,
    \item  satisfies $(g,1) \in (g,1)^\perp$.
\end{itemize}
The new element of $\mcG$ will be $(g,1)^{\perp}/(g,1)$. We will give two equivalent definitions of the process that will be useful for the sequel.

\subsubsection{Alternative definition of the process}
\label{sec: alternative definition}
We will formulate an equivalent formulation of $\Delta_S$. This fomulation will be the one that will be most important for the sequel.
\somespace
Pick a random commutative diagram:
\[
\begin{tikzcd}
&\Z_p\arrow{dr} \arrow{dl}&\\
G\arrow{dr}&&\Q_p\arrow{dl}\\
&\Q_p/\Z_p&
\end{tikzcd}
\]
such that the left side is self-dual.\footnote{This means we choice of $g \in G$ and $q\in \Q_p$ such that $<g,g>\equiv q$.} There is a unique pairing on $\Q_p$, such that the right side is also self-dual. This diagram gives rise to a self-dual sequence of homomorphisms\footnote{To get from the diagram to the sequence, we multiply the map $\Q_p \rightarrow \Q_p/\Z_p$ by $-1$; to preserve self-duality, we also multiply the pairing on $\Q_p$ by $-1$.}:
\begin{equation}
\label{eqn: sequence}
\Z_p \rightarrow \Q_p \times G \rightarrow \Q_p/\Z_p,
\end{equation}
that compose to $0$. Take the kernel modulo the image. This will be a new finite group together with a perfect pairing.

\subsubsection{Slight variation on preceding definition}
\label{sec: comparison of two processes}
We give a slight variation on the preceding definition that emphasizes that the connection to the process $\Delta_0$ in \cite{arxiv}. Recall that the generator of $\Delta_0$. In this case, given $G$, $\Delta_0(G)$ is defined by
\begin{itemize}
\item picking a uniformly random $\Z_p$-extension of $G$,
\item moddding out by a Haar random element of this extension.
\end{itemize}

\noindent
Returning to the definition of $\Delta_S$, we can formulate it as follows. We again first choose a uniformly random $\Z_p$-extension of $G$. This gives rise to an exact sequence 
\[
0\rightarrow \Z_p \rightarrow H \rightarrow G \rightarrow 0
\]
which gives rise to a commutative diagram
\[
\begin{tikzcd}
&H\arrow{dr} \arrow{dl}&\\
G\arrow{dr}&&\Q_p\arrow{dl}\\
&\Q_p/\Z_p&
\end{tikzcd}
\]
Now pick a Haar random element $\Z_p \rightarrow H$, subject to the condition that the left side of the diagram:
\[
\begin{tikzcd}
&\Z_p\arrow{dr} \arrow{dl}&\\
G\arrow{dr}&&\Q_p\arrow{dl}\\
&\Q_p/\Z_p&
\end{tikzcd}
\]
is self-dual. There will be a unique choice of pairing on $\Q_p$ that makes the right side of the diagram also self-dual. As in \S \ref{sec: alternative definition}, this gives rise to a self-dual sequence:
\[
\Z_p \rightarrow G \times \Q_p \rightarrow \Q_p/\Z_p
\]
Taking the kernel modulo the image, we get a finite abelian $p$-group together with a perfect pairing. 

\subsection{Representation by a weighted directed graph}
We will use the definition of the process as formulated in \S \ref{sec: alternative definition}. Given this formulation, it is natural to represent it as a random walk on a directed weighted graph $\Gamma$, whose \textit{vertices} are isomorphism classes in $\mcG$ and whose \textit{directed edges} are isomorphism classes of self-dual diagrams:
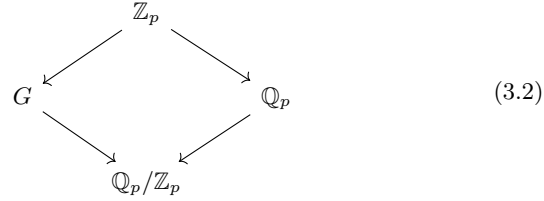
\begin{equation}
\label{eqn: self-dual diagram}
\begin{tikzcd}
&\Z_p\arrow[dr] \arrow[dl]&\\
G\arrow[dr]&&\Q_p\arrow[dl]\\
&\Q_p/\Z_p&
\end{tikzcd}
\end{equation}
where two diagrams are isomorphic if and only if they fit into a self-dual diagram:
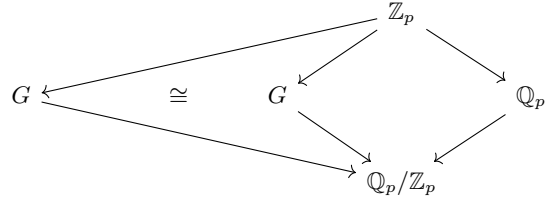
\[
\begin{tikzcd}
&&&&\Z_p\arrow[dr] \arrow[dl] \arrow[dllll]&\\
G\arrow[drrrr]&&\cong&G\arrow[dr]&&\Q_p\arrow[dl]\\
&&&&\Q_p/\Z_p&
\end{tikzcd}
\]
The weight of the edge is proportional to the probability of choosing the corresponding isomorphism class of diagrams. By the orbit-stabilizer theorem, this probability can be computed to be:
\[
\frac{\#Group}{\#G\#Stabilizer}dq=\frac{\#Aut(G,<,>)}{\#G \#Aut(G,<,>,g)}dq.
\]
We note that the total weight of the edges leaving a vertex can be chosen to be any non-vanishing measure $\mu$ on $\mcG$. Thus, the general form of the edge weight is:
\begin{equation}
\label{eqn: edgeweight}
\frac{\#Aut(G,<,>)}{\#G\#Aut(G,<,>,g)} \mu(G,<,>)dq.
\end{equation}
\paragraph{\textbf{Goal.}} Our goal is to find an end-point reversing involution on the set of directed edges, and a choice of $\mu$, such that the involution preserves the edge weight (\ref{eqn: edgeweight}). 
This will realize the random process as a symmetric random walk on an \textit{undirected} weighted graph, thereby proving that the process is reversible. \somespace Furthermore, the stationary measure will correspond to the total weight of the edges adjacent to a given vertex; hence, by definition, it will be the measure $\mu$.
\somespace
We will realize the preceding goal by first showing that the set of directed edges is in bijection with another set $\mcH$; we will see that $\mcH$ carries a natural involution $\sigma$.
Having done this, we will be find that there is a natural choice of $\mu$ such that \ref{eqn: edgeweight} is $\sigma$-invariant.

\newcommand{\todeleteone}{
\section{...}
[We begin with the self-dual diagram \ref{eqn: self-dual diagram}.]
\somespace
 We reverse the procedure in ... to get an exact sequence $0 \rightarrow \Z_p \rightarrow H \rightarrow G \rightarrow 0$, together with another map $\Z_p \rightarrow H$. [We have seen that the cokernel of this last map is a finite abelian $p$-group together with a perfect pairing.][This gives an isomorphism class of pairs of exact sequences:
 \[
 ...
 \]
 It remains to show that the compatibility condition is satisfied.]
 \somespace
 Explicitly, $H$ is the fiber product:\footnote{and there is an associated exact sequence $0\rightarrow \Z_p \rightarrow H \rightarrow G \rightarrow 0$} \footnote{Try to make this as laconic as possible.}
 \[
 H\hspace{0.2in} \cong \hspace{0.2in}
\begin{tikzcd}
G\arrow{dr}& \times & \Q_p\arrow{dl}\\
&\Q_p/\Z_p&
\end{tikzcd}
 \]
 The first map $\Z_p \xrightarrow{v_1} H$ comes from exact sequence $0\rightarrow \Z_p \rightarrow H \rightarrow G \rightarrow 0$, or equivalently the diagram:
 \[
 \begin{tikzcd}
  & \Z_p\arrow[dr,"1"] \arrow[dl,"0"']&       \\
G \arrow[dr]&          &\Q_p \arrow[dl]     \\
           &\Q_p/\Z_p &                    \\
\end{tikzcd}
 \]
 
 The second map $\Z_p \xrightarrow{v_2} H $ comes from the self-dual diagram \ref{eqn: self-dual diagram}.
\somespace
$H$ has two pairings, the old pairing and the new pairing. By construction $v_1$ generates the kernel of the old pairing and $v_2$ generates the kernel of the new pairing. The difference of the two pairings factors through the map $H \rightarrow \Q_p$. [It can be deduced that] the two pairings satisfy the following compatibility condition:
\begin{equation}
\label{eqn: compabitility condition}
<h_1,h_2>_2-<h_1,h_2>_1 \hspace{0.1in} \equiv \hspace{0.1in} \frac{(h_1\otimes \Q_p)(h_2\otimes \Q_p)}{(v_1\otimes \Q_p)(v_2\otimes \Q_p)} \mod \Z_p \hspace{0.1in}
\end{equation}
for all $h_1,h_2\in H$.

 }

\subsection{Definition of $\mcH$}

\newcommand{\additionalcommentsone}{
To prove that the random process is reversible, we need to represent it as a random walk on an undirected graph. To do so, we will find a measure preserving involution on the directed edges that reverses end-points. This involution will preserve the weights, for some carefully chosen $F$.
\somespace
To define this involution, we parametrize the set of directed edges by another set $\mcH$. The advantage of this is that $\mcH$ carries a natural involution.
First, we define $\mcH$.
}

\paragraph{Definition}
Let $\mcH$ to be the set of isomorphism classes of finitely generated abelian $p$-groups $H$ of $\Q_p$-rank $1$, together with the following data:
\begin{itemize}
    \item Two elements of $H$, denoted by $v_1$ and $v_2$, 
    \item A perfect symmetric pairing on $H/v_1$, and a perfect symmetric pairing on $H/v_2$,
\end{itemize}
satisfying the following compatibility condition:
\begin{equation}
\label{eqn: relation}
<\sdot,\sdot>_1-<\sdot,\sdot>_2
\, \equiv \,
\frac{(\sdot \otimes\Q_p)(\sdot \otimes\Q_p)}{(v_1 \otimes\Q_p)(v_2 \otimes\Q_p)}
\mod \Z_p \end{equation}
We denote an element of $\mcH$ as $(H,v_1,v_2,<,>_1,<,>_2)$.

\paragraph{Comment} We explain the notation in (\ref{eqn: relation}). Since $H$ has $\Q_p$ rank $1$, there is a homomorphism $H \rightarrow \Q_p$, given by $H \rightarrow H \otimes\Q_p \cong \Q_p$. This homomorphism is unique up to $Aut(\Q_p)$, i.e. up to multiplication by $\Q_p^{*}$. Hence, the expression on the right of (\ref{eqn: relation}) is well-defined. 

\begin{corollary} 
Substituting $v_1$ and $v_2$ into the expression (\ref{eqn: relation}), we find:\begin{align*}
<\sdot,v_2>_1 \,&=\,\,\frac{(\sdot\otimes \Q_p)}{(v_1\otimes \Q_p)}\\
<\sdot,v_1>_2\,&= -\,\frac{(\sdot\otimes \Q_p)}{(v_2\otimes \Q_p)}
\end{align*}
\end{corollary}

\noindent
Diagrammatically, we represent an element of $\mcH$ as:
\[
\begin{tikzcd}
&&0\arrow[d]&&\\
&&\Z_p\arrow[d,"v_2"']&&\\
0\arrow[r]&\Z_p\arrow[r,"v_1"]&H \arrow[r] \arrow[d]&H/v_1\arrow[r]&0\\
&&H/v_2\arrow[d]&&\\
&&0&&\\
\end{tikzcd}
\]

\subsection{Correspondence between $\mcH$ and (directed) edges of $\Gamma$}
\label{sec: correspondence g and h}
Our aim in this section is to show that isomorphism classes in $\mcH$ parametrized the directed edges of $\Gamma$. Recall that the edges of $\Gamma$ are parametrized by commutative diagrams of the form:
\begin{equation}
\label{eqn: commutative diagram one}
 \begin{tikzcd}
  & \Z_p\arrow[dl,"g"'] \arrow[dr,"q"]&       \\
G \arrow[dr,"<g\comma \sdot>"']&          &\Q_p \arrow[dl,"1"]     \\
           &\Q_p/\Z_p &                    \\
\end{tikzcd}
\end{equation}
Thus, we need to establish a bijection between isomorphism classes of diagrams of the form, and elements of $\mcH$; we will in fact see that this bijection will give isomorphisms between groupoids of isomorphic elements.
\subsubsection{First direction}
From the diagram (\ref{eqn: commutative diagram one}), we can obtain the diagram below, in a natural way:
\begin{equation}
\label{eqn: intermediate diagram}
 \begin{tikzcd}
 &&0\arrow[d]&&\\
 & & \Z_p\arrow[d, "v_2"'  ]&     &  \\
0 \arrow[r]& \Z_p \arrow[r, "v_1"] &    ker\left(
\begin{array}{c}
G \times \Q_p\\
\downarrow\\
\Q_p/\Z_p\\
\end{array}
\right) \arrow[r]&G \arrow[r]& 0\\
\end{tikzcd}
\end{equation}
\noindent
where 
\begin{equation} 
\label{eqn: vone and vtwo def} v_2=(g,q) \text{ and } v_1=(0,1). \end{equation} Via the definition of the process as formulated in \S \ref{sec: alternative definition}, we recall that the cokernel of the homomorphism $v_2$ carries a natural perfect pairing. To verify that we obtain an element of $\mcH$, it remains to verify the compatibility condition (\ref{eqn: relation}).

\paragraph{Proof of compatibility condition.} We note that in order for the diagram (\ref{eqn: commutative diagram one}) to be self-dual, we need the pairing on $\Q_p$ to be $(q_1,q_2) \rightarrow \frac{q_1 q_2}{q} \mod \Z_p$, So the pairing is given by $\frac{1}{q}$. When we pass from commutative diagrams to sequences of the form 4.2, we need to twist the map $\Q_p \rightarrow \Q_p/\Z_p$ by $-1$; hence we also need to twist the pairing by $-1$ to preserve self-duality. Note that the difference of the pairings:
\[
<\sdot, \sdot>_2 - <\sdot, \sdot>_1
\]
is precisely the pairing on the $\Q_p$ factor, and is hence given by $-\frac{1}{q} \mod \Z_p$. By (\ref{eqn: vone and vtwo def}), we can deduce that the compatibility condition is indeed satisfied.

\subsubsection{Other direction}
From the diagram,
\begin{equation*}
\begin{tikzcd}
&&0\arrow[d]&&\\
&&\Z_p\arrow[d,"v_2"']&&\\
0\arrow[r]&\Z_p\arrow[r,"v_1"]&H \arrow[r] \arrow[d]&H/v_1\arrow[r]&0\\
&&H/v_2\arrow[d]&&\\
&&0&&\\
\end{tikzcd}
\end{equation*}
we get the following diagram, in which both compositions are $0$:
\begin{equation}
\label{eqn: vone and vtwo}
\begin{tikzcd}
 &\Z_p \arrow[d,"v_2"']&   \\
 \Z_p \arrow[r,"v_1"]&(H/v_1)  \times \Q_p \arrow[r]&\Q_p/\Z_p   \\
\end{tikzcd}
\end{equation}

\noindent
The map on the right is given by: \[
(h,a) \longmapsto\left( 
\,
\frac{h \otimes \Q_p}{v_1 \otimes \Q_p}
\,
,
\,
-a \,
\right)
\mod \Z_p
\]
The compatibility condition applied to $(v_2, \cdot)$ shows that the following pair of maps are dual to one another:
\[
\begin{tikzcd}
 \Z_p \arrow[d,"v_2"']&   \\
H/v_1 \arrow[r]&\Q_p/\Z_p   \\
\end{tikzcd}
\]
Therefore, the top and the rightmost homomorphisms of $\ref{eqn: vone and vtwo}$ give rise to the commutative diagram:
\[
\begin{tikzcd}
  &\Z_p \arrow[dr] \arrow[dl,"v_2"']&  \\
  H/v_1\arrow[dr," <\,v_2 \, \comma \, \sdot \,>"']&&\Q_p\arrow[dl]  \\
  &\Q_p/\Z_p&  \\
\end{tikzcd}
\]


\subsection{Further properties of $\mcH$}

\subsubsection{Endpoints of edge parametrized by an element of $\mcH$}
 There are two natural maps $\pi_i: \mcH \rightarrow \mcG$, given by modding out by $v_i$. One can verify that these two maps are the end-points of the edge corresponding to $\mcH$.

\subsubsection{Ratio of $\#(H/v_1)$ and $\#(H/v_2)$}
\paragraph{Proposition}
We have the identity:
\[
\frac{\#(H/v_1)}{\#(H/v_2)}=\left| \frac{v_2\otimes \Q}{v_1 \otimes \Q} \right|_p
\]
\begin{proof} First of all, the preceding holds, if we assume that $v_1 = av_2$ for some $a \in \Z_p$. In that case, on the left, we get $\#(\Z_p/a\Z_p)$ and on the right we get $\left| a^{-1}\right|_p$, which are equal. 
\somespace
More generally, we use the following identity.
\[
\frac{\#(H/v_1)}{\#(H/v_2)}=p^{-N}\frac{\#(H/p^N v_1)}{\#(H/v_2)}
\]
For $N$ large enough, $p^N v_1$ lies in the span of $v_2$. Therefore, there exists $a \in \Z_p$ such that $p^N v_1=av_2$, and the result follows as before. 
\end{proof}

\paragraph{Corollary}

\begin{equation}
\label{eqn: important identity one}
\frac{\#(H/v_1)}{\#(H/v_2)}=\left| q \right|_p
\end{equation}

\subsubsection{An isomorphism of automorphism groups}

%

\paragraph{Proposition}
Under the correspondence of \S \ref{sec: correspondence g and h}, \begin{equation}
\label{eqn: important identity two}
Aut(H,<,>_1,<,>_2,v_1,v_2) \cong Aut(G,<,>,g).\end{equation}
If an isomorphism $H/v_1 \cong G$ is given, then the isomorphism is canonical.

\begin{proof}
Pick an isomorphism $g: H/v_1 \cong G$. Note that the map $\Z_p \xrightarrow{v_1} H$ induces a homomorphism $q:H\rightarrow \Q_p$, by tensoring with $\Q_p$ and normalizing so that the image of $v_1$ is $1 \in \Q_p$. Every element $h \in H$ is uniquely determined by the pair $(g(h),q(h))$. $g(h)$ and $q(h)$ must satisfy the compatibility condition:
\[
\phi(g)\equiv q \mod \Z_p
\]
Conversely, any pair of elements $(g,q)$ satisfying this condition gives an element of $H$. Let $\tau \in Aut(H,v_1)$. Then
\begin{itemize}
\item $q(\tau(h))=h$
\item $g(\tau (h))=(\tau g)(h)$
\item $\phi(g(\tau(h)))\equiv q(\tau(h))\equiv q(h)\equiv\phi(g(h))$
\end{itemize}
This shows the following lemma:
\paragraph{Lemma} $Aut(H,v_1) \cong Aut(G,\phi)$
Under this isomorphism, we have 
\begin{itemize}
\item Elements of $Aut(H,v_1)$ that preserve a given element $w \in H$ correspond to elements of $Aut(G,\phi)$ that preserve the image of $w$ in $G$.
\item Elements of $Aut(H,v_1)$ that preserve the pairing $<,>_1$ are precisely those elements of $Aut(G,\phi)$ that preserve $<,>$.
\end{itemize}

The preceding discussion implies that elements of $Aut(H,v_1,v_2,<,>_1)$ are mapped to elements of $Aut(G,\phi)$ that preserve the image of $v_2$ in $G$, and the pairing $<,>$ on $G$. Observing that, by construction $\phi = <g,\sdot>$, we find that an element of $Aut(G,\phi)$ is in the image of $Aut(H,v_1,v_2,<,>_1)$, if and only if it is contained in $Aut(G,g,<,>)$. Hence,
\begin{equation}    
\label{eqn: isomorphismtwo}
Aut(H,v_1,v_2,<,>_1) \cong Aut(G,g,<,>)
\end{equation}
\paragraph{Lemma}
\begin{equation}
\label{eqn: isomorphismone}
Aut(H,v_1,v_2,<,>_1)=Aut(H,v_1,v_2,<,>_1,<,>_2)
\end{equation}
This lemma from the identity, \ref{eqn: relation}, relating $<,>_1$ and $<,>_2$.
\newline
\newline
\noindent
Combining (\ref{eqn: isomorphismone}) and (\ref{eqn: isomorphismtwo}) , we get  
\[
Aut(H,v_1,v_2,<,>_1, <,>_2) \cong Aut(G,g,<,>)
\]
\end{proof}

\subsection{An involution $\sigma$ on isomorphism classes in $\mcH$}
\label{sec: involution}

Isomorphism classes in $\mcH$ carry a natural involution given by
\[
v_1 \mapsto v_2 \hspace{0.5in} v_2 \mapsto -v_1 \hspace{0.5in}
<\,,\,>_1 \, \mapsto \, <\,,\,>_2 \hspace{0.5in}
<\,,\,>_2 \, \mapsto \, <\,,\,>_1 \hspace{0.5in}
\]
We note that $\sigma$ preserves the compatibility conditions on $v_1,v_2,<,>_1,<,>_2$, hence we indeed a map $\mcH \rightarrow \mcH$. Moreover, if we apply $\sigma$ twice, we get \[(H,-v_1,-v_2,<,>_1,<,>_2),\] which lies in the same isomorphism class as \[(H,v_1,v_2,<,>_1,<,>_2).\] Thus we indeed get an involution.
\somespace
Moreover, we note that the group $Aut(H,v_1,v_2,<,>_1, <,>_2) \subset Aut(H)$ is preserved by $\sigma$.
\somespace
We also note that $\pi_1 \circ \sigma = \pi_2$, $\pi_2 \circ \sigma = \pi_1$. Hence, if we interpret $\mcH$ as directed edges, then $\sigma$ reverses endpoints.
\somespace
Finally, we make the following important observation:
\begin{equation}
\label{eqn: important property three}
\sigma\text{ takes }q \text{ to }\frac{-1}{q}.
\end{equation}

\subsection{A choice of $\mu$ such that $\sigma$ is measure-preserving.}

First, of all, we can rewrite the weight of the edge corresponding to \[(H,v_1,v_2,<,>_1, <,>_2).\]
Using \ref{eqn: edgeweight}, as well as \ref{eqn: important identity one}, \ref{eqn: important identity two} and the isomorphism $G \cong H/v_1$, we find that the weight is given by:
\[
\mu(G,<,>)\frac{\#Aut(G,<,>)}{\#G \#Aut(G,<,>,g)} |dq|=\]
\[
=\frac{\mu(G,<,>)\#Aut(G,<,>)}{\#(H/v_1) \#Aut(H,v_1,v_2,<,>_1, <,>_2)} |dq|=
\]
\begin{equation}
\label{eqn: final expression}
=\frac{\mu(G,<,>)\#G\#Aut(G,<,>)}{\#(H/v_1) \#(H/v_2) \#Aut(H,v_1,v_2,<,>_1, <,>_2)} \frac{|dq|}{|q|}
\end{equation}
Now note that the term $\frac{|dq|}{|q|}$ is invariant under the transformation $q \rightarrow -\frac{1}{q}$; hence, by (\ref{eqn: important property three}), it is invariant by the action of $\sigma$. Furthermore, the denominator of (\ref{eqn: final expression}) is invariant under $\sigma$.
\somespace
Hence, only the numerator $\mu(G,<,>)\#G\#Aut(G,<,>)$ is not necessarily invariant under $\sigma$. If we choose 
\[
\mu(G,<,>) \propto \frac{1}{\#G\#Aut(G,<,>)}
\]
then the numerator of \ref{eqn: final expression} will be constant and \ref{eqn: final expression}  will be invariant under the involution $\sigma$. 
\somespace
If we identify directed edges by this involution, we get an undirected weighted graph. 
\somespace
It follows that the Markov chain is reversible with respect to the measure $\mu$. Indeed a symmetric random walk on a weighted graph is reversible with respect to the measure given by taking the sum of the weights of edges adjacent to a vertex. But this is precisely $\mu$.



\subsection*{Acknowledgements}

The author thanks Manjul Bhargava for suggesting the question that led to the results described in this paper. Secondly, the author thanks Yakov G. Sinai for introducing the author to renormalization group methods in probability; these ideas inspired the Markov chain approach pursued here. The author thanks Yifeng Huang, Hoi Nguyen, Alexander Van Werde and Shiqiao Zhang for fruitful discussions related to symmetric matrices and Markov chain methods. The author thanks Roger Van Peski for discussions related to \cite{shen}. The author thanks Xi Sisi Shen for encouragement during the uncertain early stages of this research project. Finally, the author thanks Alexander Yu for his encouragement and feedback throughout the last decade of the author's mathematical journey.
\newline
\newline
\noindent
This is an elaboration of the author's thesis work. AI was not used in this project.

\bibliographystyle{alpha}
\bibliography{ThesisBibliographyPrivetPrivet.bib}

\end{document}